\documentclass{amsart}

\usepackage{comment}
\usepackage{graphicx,amssymb,amsthm}
\usepackage{color}
\newtheorem{theorem}{Theorem}[section]
\newtheorem{corollary}[theorem]{Corollary}
\newtheorem{proposition}[theorem]{Proposition}

\theoremstyle{remark}

\theoremstyle{definition}
\newtheorem{example}{Example}[section]

\title{Chirally cosmetic surgeries and Casson invariants}

\author{Kazuhiro Ichihara}
\address{Department of Mathematics, College of Humanities and Sciences, Nihon University, 3-25-40 Sakurajosui, Setagaya-ku, Tokyo 156-8550, JAPAN}
\email{ichihara@math.chs.nihon-u.ac.jp}
\author{Tetsuya Ito}
\address{Department of Mathematics, Kyoto University, Kyoto 606-8502, JAPAN}
\email{tetitoh@math.kyoto-u.ac.jp}
\author{Toshio Saito}
\address{Department of Mathematics, Joetsu University of Education, 1 Yamayashiki, Joetsu 943-8512, JAPAN}
\email{toshio@juen.ac.jp}
% Subject code(s)

% Key word(s)

\begin{document}

\begin{abstract}
We study chirally cosmetic surgeries, that is, a pair of Dehn surgeries on a knot producing homeomorphic 3-manifolds with opposite orientations. 
Several constraints on knots and surgery slopes to admit such surgeries are given. 
Our main ingredients are the original and the $SL(2,\mathbb{C})$ version of Casson invariants. 
As applications, we give a complete classification of chirally cosmetic surgeries on alternating knots of genus one.
\end{abstract}

%\pagewiselinenumbers

\maketitle

\section{Introduction}

Given a knot, one can produce by Dehn surgeries a wide variety of 3-manifolds. 
Generically `distinct' surgeries on a knot, meaning that surgeries along inequivalent slopes, can give distinct 3-manifolds. 
In fact, Gordon and Luecke proved in \cite{GordonLuecke} that, on a non-trivial knot in the 3-sphere $S^3$, 
any Dehn surgery along a non-meridional slope never yields $S^3$, while the surgery along the meridional slope always gives $S^3$. 
However, it sometimes happens that `distinct' surgeries on a knot $K$ give rise to homeomorphic 3-manifolds: 
\begin{itemize}
\item[(i)] When $K$ is amphicheiral, for every non-meridional, non-longitudinal slope $r$, 
the $r$-surgery and the $(-r)$-surgery always produce 3-manifolds which are orientation-reversingly homeomorphic to each other.
\item[(ii)] When $K$ is the $(2,n)$-torus knot, there is a family of pairs of Dehn surgeries that yield orientation-reversingly homeomorphic 3-manifolds, 
first pointed out by Mathieu in \cite{Mathieu90, Mathieu92}. 
For example, $(18k+9)/(3k+1)$- and $(18k+9)/(3k+2)$-surgeries on the right-handed trefoil in $S^3$ yield orientation-reversingly homeomorphic 3-manifolds 
for any non-negative integer $k$. 
\end{itemize}
Notations used above will be given later in this section. In view of this, we say that a pair of Dehn surgeries are \emph{purely cosmetic} 
if two surgeries give the orientation-preservingly homeomorphic 3-manifolds, and \emph{chirally cosmetic} if they give the orientation-reversingly homeomorphic 3-manifolds.

The famous Cosmetic Surgery Conjecture states that there are no purely cosmetic surgeries along inequivalent slopes. 
See \cite[Problem 1.81(A)]{Kirby} for further information and more precise formulation.
Recent progress has been made on the conjecture, including supporting evidence, such as a striking result obtained by Ni and Wu \cite{NiWu}.

On the other hand, as for chirally cosmetic surgeries, the situation is more subtle and complicated as the aforementioned Mathieu's example suggests. 
The examples of Mathieu are generalized to the knots with exteriors which are Seifert fibered spaces by Rong in \cite{Rong}, 
and to the non-hyperbolic knots in lens spaces by Matignon in \cite{Matignon}, 
where a classification of chirally cosmetic surgeries on the non-hyperbolic knots in lens spaces is achieved. 
Further examples of hyperbolic knots yielding lens spaces are given by Bleiler, Hodgson and Weeks in \cite{BleilerHodgsonWeeks}, 
and also, of hyperbolic knots yielding hyperbolic manifolds are given by Ichihara and Jong in \cite{IchiharaJong}.

In this paper, we give several constraints on knots and surgery slopes to admit chirally cosmetic surgeries coming from various invariants of closed oriented 3-manifolds. 
In particular, we will extensively discuss constraints arising from the Casson invariant and the $SL(2,\mathbb{C})$ version of the Casson invariant.

In Section~\ref{sec:Casson}, we will use the Casson invariant and the Casson-Gordon invariant to give a constraint on knots to admit chirally cosmetic surgeries 
(Theorem~\ref{theorem:cassonobst}). 
As a corollary, we discuss chirally cosmetic surgeries along slopes with small numerators (Corollary~\ref{cor:smallp}) and the parity of chirally cosmetic surgery slopes 
(Corollary~\ref{cor:Casson2}).

In Section~\ref{sec:SL(2,C)Casson}, after reviewing a surgery formula of $SL(2,\mathbb{C})$ Casson invariant, 
we give another constraint on  chirally cosmetic surgeries (Theorem~\ref{theorem:SL2Ccasson-1}). 
In Section~\ref{sec:bdryslope}, we discuss more geometric formulation that relates the boundary slopes and cosmetic surgeries 
(Theorems~\ref{theorem:boundaryslopeI} and \ref{theorem:boundaryslopeII}). 

In Section \ref{section:others}, we review other known obstructions for knots to admit chirally cosmetic surgeries 
such as an obstruction from the degree two part of the LMO invariant, and several results from Heegaard Floer theory.

{
In Section \ref{section:application}, by combining the techniques discussed so far, we give a criterion for non-existence of chirally cosmetic surgeries.

Let $a_2(K)$ and $a_4(K)$ be the coefficient of the Conway polynomial $\nabla_K(z)$ and $v_{3}(K)= -\dfrac{1}{144}V'''_{K}(1)-\dfrac{1}{48}V''_{K}(1)$ where $V_K(t)$ denotes the Jones polynomial of $K$.

\setcounter{section}{6}
\setcounter{theorem}{0}
\begin{theorem}
Let $K$ be a knot and let $d(K)$ be the degree of the Alexander polynomial of $K$. If $S^{3}_{K}(p\slash q) \not\cong -S^{3}_{K}(p\slash q')$ for some $p/q,p/q'$ with $q+q'\neq 0$, then $v_3(K)\neq 0$ and
\[ 4|a_2(K)| \leq d(K) \left| \frac{7a_2^{2}(K)-a_2(K)-10a_4(K)}{8v_3(K)} \right|\]
\end{theorem}

\begin{corollary}
Let $K$ be a non-trivial positive knot. If $\frac{16}{7} \geq 7g(K)\frac{a_2(K)}{v_3(K)}$ then $K$ does not admit chirally cosmetic surgery, where $g(K)$ denotes the genus of $K$.
\end{corollary}

As an application, we show that knots obtained by the $\overline{t}_{N}$-move (see Section \ref{section:application} and Figure \ref{fig2} for the definition of $\overline{t}_{N}$-move) with sufficiently large $N$ have no chirally cosmetic surgery.

\begin{theorem}
Let $K$ be a knot represented by a diagram $D$.
At a positive crossing $c$ of $D$, let $L=K' \cup K''$ be the 2-component link obtained by resolving the crossing $c$, and for $N>0$, let $K_{N}$ be the knot obtained from $K$ by applying $\overline{t}_{N+1}$-move at $c$, which replaces the positive crossing $c$ with consecutive $(2N+1)$ positive crossings.
If $lk(K',K'')\neq 0$, then for a sufficiently large $N$, $K_N$ does not admit chirally cosmetic surgery.
\end{theorem}

Finally we completely classify all the chirally cosmetic surgeries on alternating knots of genus one. 

\setcounter{section}{6}
\setcounter{theorem}{2}
\begin{theorem}
Let $K$ be an alternating knot of genus one. 
For distinct slopes $r$ and $r'$, if the $r$- and $r'$-surgeries on $K$ are chirally cosmetic, 
then either
\begin{enumerate}
\item[(i)] $K$ is amphicheiral and $r=-r'$, or 
\item[(ii)] $K$ is the positive or the negative trefoil, and 
\[ \{r,r'\} = 
\left\{\frac{18k+9}{3k+1}, \frac{18k+9}{3k+2} \right\},\ \mathit{or}\ 
\left\{- \frac{18k+9}{3k+1}, - \frac{18k+9}{3k+2} \right\} \quad (k \in \mathbb{Z}).\]
\end{enumerate}
\end{theorem}

}
\setcounter{section}{1}

Our results give an affirmative answer to the following question raised in \cite{ItoTetsuya}, for the case of alternating knots of genus one:
\emph{If $r$- and $(-r)$- surgeries yield orientation-reversingly homeomorphic 3-manifolds for some $r \ne 0, 1/0$, then is $K$ amphicheiral?}

In particular, they also give a supporting evidence for the following much stronger and optimistic question:
\emph{If a knot $K$ in $S^{3}$ admits chirally cosmetic surgeries, then is $K$ either the $(2,n)$-torus knot, or an amphicheiral knot?}

\bigskip

In the rest of the introduction, we recall basic definitions and terminologies about Dehn surgery which we will use in this paper.  

For an oriented closed 3-manifold $M$, we denote by $-M$ the same 3-manifold with opposite orientation. 
We denote $M\cong M'$ if two 3-manifolds $M$ and $M'$ are orientation-preservingly homeomorphic, and $M\cong \pm M'$ if $M$ is orientation-preservingly homeomorphic to $M'$ or $-M'$. 

Let $K$ be a knot in a 3-manifold $M$, and let $E(K)$ be the exterior of $K$, i.e., the complement of an open tubular neighborhood of $K$ in $M$. 
A \emph{slope} is an isotopy class of a non-trivial unoriented simple closed curve on the boundary $\partial E(K)$. 

When $K$ is a knot in $S^{3}$, or, more generally, $K$ is a null-homologous knot in a rational homology sphere, 
the set of slopes are identified with $\mathbb{Q} \cup \{\infty = \frac{1}{0}\}$ in the following manner. 
For a slope $\gamma$, we consider the element $[\gamma] = p[\mu]+q[\lambda] \in H_{1}(\partial E(K);\mathbb{Z})$ represented by $\gamma$. 
Here $p$ and $q$ are coprime integers, $\mu$ denotes the meridian and $\lambda$ the preferred longitude of the knot $K$. 
Then we assign the rational number $r=\frac{p}{q}$ (possibly $\infty=\frac{1}{0}$) to represent the slope $\gamma$.
Throughout the paper, when we express a slope as a rational number $\frac{p}{q}$, we will always take $p\geq 0$ and $\frac{p}{q}$ is irreducible. 
For such a rational number $r$, let $M_{K}(r)$ be the closed oriented 3-manifold obtained by attaching a solid torus to $E(K)$ 
so that the closed curve of the slope $r$ bounds a disk in the attached solid torus, namely, the slope $r$ is identified with the meridional slope of the attached solid torus. 
We call $M_{K}(r)$ \emph{the 3-manifold obtained by Dehn surgery on $K$ along $r$}, or simply the \emph{$r$-surgery} on $K$.

\section{Casson-Walker invariant and Casson-Gordon invariant}\label{sec:Casson}

In this section, we use the Casson-Walker invariant, together with the (total) Casson-Gordon invariant, 
to give constraints on knots and surgery slopes to admit chirally cosmetic surgeries. 
Such an approach to make use of the Casson invariant was originally developed by Boyer and Lines in \cite{BoyerLines}, and further studied by Ni and Wu in \cite{NiWu}, 
where the Heegaard Floer theory was also applied.

Let $\lambda$ be the Casson-Walker invariant for rational homology 3-spheres. 
See \cite{Walker} for its definition and basic properties. 
For a rational homology sphere $M$ with $H_{1}(M;\mathbb{Z})$ which is finite cyclic, namely, a homology lens spaces, it is also defined in \cite{BoyerLines}. 
Also, let $\tau$ be the (total) Casson-Gordon invariant of homology lens space. See \cite[Definition 2.20]{BoyerLines} and \cite{CassonGordon} for the definition.

Let $K$ be a knot in an integral homology 3-sphere $\Sigma$. 
The Casson-Walker invariant $\lambda$ and the Casson-Gordon invariant $\tau$ satisfy the following surgery formulae \cite{BoyerLines,Walker}:  
\[ \lambda(\Sigma_{K}(p\slash q)) =  \lambda(\Sigma) + \frac{q}{p} a_{2}(K) - \frac{1}{2}s(q,p), \ \  \tau(\Sigma_{K}(p\slash q)) = -4p \cdot s(q,p) - \sigma(K,p)
%+\sigma(K,p).
 \]
Here we give definitions of the terms $a_{2}(K)$, $s(q,p)$, and $\sigma(K,p)$ used in the formulae. 
First  $a_2 (K)$ denotes $\frac{1}{2}\Delta''(1)$, the second derivative of the Alexander polynomial $\Delta_{K}(t)$ of $K$ at $t=1$, 
where the Alexander polynomial is normalized so that $\Delta(t)=\Delta(t^{-1})$ and that $\Delta(0)=1$. 
In the case $\Sigma=S^{3}$, it coincides to the second coefficient of the Conway polynomial $\nabla_{K}(z)$ of $K$. 
See \cite{AkbulutMcCarthy} for example. 
 
For coprime integers $p,q$ with $p>0$, $s(q,p)$ denotes the \emph{Dedekind sum} defined by
\[ s(q,p) = \sum_{k=1}^{p-1}\Bigl(\!\!\Bigl( \frac{k}{p}\Bigr)\!\!\Bigr)\Bigl(\!\!\Bigl( \frac{kq}{p}\Bigr)\!\!\Bigr) \]
with $(\!(x)\!) = x -\lfloor x \rfloor -\frac{1}{2}$ and the floor function $\lfloor x \rfloor$ (the maximum integer which is less than or equal to $x$) for $x \in \mathbb{Q}$. 

Also $\sigma(K,p)= \sum_{\omega: \omega^{p}=1}\sigma_{\omega}(K)$ denotes \textit{the total $p$-signature} for $K$, where $\sigma_{\omega}(K)$ ($\omega \in \{z \in \mathbb{C}\: | \: |z|=1 \}$) denotes \textit{the Levine-Tristram signature}, i.e., the signature of $(1-\omega)S + (1-\overline{\omega})S^{T}$ for a Seifert matrix $S$ of $K$.

The Casson invariant has the property $\lambda(-M)=-\lambda(M)$. 
Also, we see that $\tau(M)=-\tau(M)$ by the definition of the Casson-Gordon invariant \cite[Definition 2.20]{BoyerLines}. 
Together with these, the surgery formulae above imply the following. 

\begin{theorem}
\label{theorem:cassonobst}
Let $K$ be a knot in an integral homology 3-sphere $\Sigma$. 
If $ \Sigma_{K}(p\slash q) \cong -\Sigma_{K}(p\slash q')$, 
then 
\[ 12(q+q')a_{2}(K)+ 24p \lambda(\Sigma)= 6p ( s(q,p) + s(q',p) ) = - 3\sigma(K,p)\]
%3\sigma(K,p)\]
holds. 
\end{theorem}

\begin{proof}
If $ \Sigma_{K}(p\slash q) \cong -\Sigma_{K}(p\slash q')$, then we have the following by the surgery formulae above. 
\[\begin{cases}
 \lambda(\Sigma_{K}(p\slash q)) - \left( - \lambda(\Sigma_{K}(p\slash q')) \right) = 2\lambda(\Sigma) +\frac{q+q'}{p} a_{2}(K) - \left( \frac{1}{2}s(q,p) + \frac{1}{2}s(q',p) \right) = 0, \\
\tau(\Sigma_{K}(p\slash q)) - \left( - \tau( \Sigma_{K}(p\slash q') ) \right) = -4 p \left( s(q,p) + s(q',p) \right)  -2 \sigma(K,p)=0.
%+ 2 \sigma(K,p) = 0. 
\end{cases}
\]
These imply the equalities which we want. 
\end{proof}

In the rest of this section, we show some corollaries to this theorem.

\subsection{Constraints on surgery slopes}

By checking small values for the numerator $p$ of the surgery slope for the case $\Sigma=S^{3}$, or more generally, an integral homology sphere $\Sigma$ with $\lambda(\Sigma)=0$, we obtain the following more explicit constraints on surgery slopes from Theorem~\ref{theorem:cassonobst}.

\begin{corollary}\label{cor:smallp}
Let $K$ be a knot in an integral homology 3-sphere $\Sigma$  with $\lambda(\Sigma)=0$, and $p$ an integer with $1 \leq p \leq 10$. 
If $ \Sigma_{K}(p\slash q) \cong -\Sigma_{K}(p\slash q')$, 
then one of the following holds.
\begin{enumerate}
\item $a_{2}(K)=0$ or $q=-q'$, and $\sigma(K,p)=0$.
\item $p=7$, $q=7s+1$, $q'=-7s-2$  $(s\in \mathbb{Z})$, $a_{2}(K)=-1$ and $\sigma(K,7)=-4$.

\item $p=7$, $q=7s+2$, $q'=-7s-1$  $(s\in \mathbb{Z})$, $a_{2}(K)=-1$ and $\sigma(K,7)=4$.

\item $p=9$, $q=1+9s$, $q'=2-9s$  $(s\in \mathbb{Z})$, $a_{2}(K)=1$ and 
$\sigma(K,9)=-12$.

\item $p=9$, $q=-1+9s$, $q'=-2-9s$ $(s\in \mathbb{Z})$, $a_{2}(K)=1$ and  $\sigma(K,9)=12$.

\item $p=9$, $q=1+9s$, $q'=-4-9s$  $(s\in \mathbb{Z})$, $a_{2}(K)=-1$ and $\sigma(K,9)=-12$.

\item $p=9$, $q=-1+9s$, $q'=4-9s$  $(s\in \mathbb{Z})$, $a_{2}(K)=-1$ and $\sigma(K,9)=12$.
\end{enumerate}
\end{corollary}
\begin{proof}
We start with considering the case of $p=7$. 
Then the Dedekind sum is given as in Table~\ref{table0} below. 
Hence $42s(q,7)+42s(q',7) \in \{0,\pm 6,\pm 12, \pm 18,\pm 30\}$.

\begin{table}[htbp]
\centering
\begin{tabular}{|c|c|c|c|c|c|c|} \hline
$q$ mod 7& 1 & 2 & 3 & 4 & 5 & 6  \\ \hline 
42$s(q,7)$&15 & 3 & -3 & 3 & -3 & -15 \\ \hline
  \end{tabular}
\smallskip
\caption{Dedekind sum 42$s(q,7)$}\label{table0}
\end{table}

It follows from Theorem~\ref{theorem:cassonobst} that 
%$12a_{2}(K)(q+q')=6p(s(q,p)+s(q',p))$. 
$12a_{2}(K)(q+q')=42s(q,7)+42s(q',7)$
so $42s(q,7)+42s(q',7)$ must be divisible by $12$. 
Thus, if $a_{2}(K)\neq 0$ and $q+q' \neq 0$, then $42s(q,7)+42s(q',7)= \pm 12$. Consequently, we have
\begin{itemize}
\item $a_{2}(K)= \pm 1$, $q+q'= \pm 1$ and  $(q,q') \equiv (1,3),(1,5),(2,6),(4,6) \pmod 7$. 
\end{itemize}
This gives rise to Case 2.

We next explain the most non-trivial case $p=9$, where the exceptional possibilities Cases 4--7 appears. 
For $p=9$, the Dedekind sum is calculated as in Table~\ref{table1} below, so  $54s(q,9)+54s(q',9) \in \{0,\pm8,\pm16,\pm20,\pm36,\pm56\}$.

\begin{table}[htbp]
\centering
\begin{tabular}{|c|c|c|c|c|c|c|} \hline
$q$ mod 9& 1 & 2 & 4 & 5 & 7 & 8  \\ \hline 
54$s(q,9)$&28 & 8 & $-8$ & 8 & $-8$ & $-28$ \\ \hline
  \end{tabular}
\smallskip
\caption{Dedekind sum $s(q,9)$}\label{table1}
\end{table}

Thus if $a_{2}(K)\neq 0$ and $q+q' \neq 0$, then either
\begin{itemize}
\item $a_{2}(K) = \pm 1$, $q+q'=\pm 3$, and $(q,q') \equiv (1,2),(1,5),(4,8),(7,8) \mod 9$, or,
\item $a_{2}(K) = \pm 3$, $q+q'=\pm 1$, and $(q,q') \equiv (1,2),(1,5),(4,8),(7,8) \mod 9$. 
\end{itemize}
The first case gives Cases 4--7. 
The second possibility cannot happen, for $q+q' \not \equiv \pm 1 \mod 9$. 
\end{proof}

Note that Cases 4 and 5 correspond to chirally cosmetic surgeries of the right-handed and left-handed trefoil, respectively. 

We also have another corollary to Theorem~\ref{theorem:cassonobst} which gives a restriction on the parity of the numerator $p$ of the surgery slopes. 

\begin{corollary}
\label{cor:Casson2}
Let $K$ be a knot in an integral homology 3-sphere $\Sigma$ with signature $\sigma(K)$. 
Assume that $ \Sigma_{K}(p\slash q) \cong -\Sigma_{K}(p\slash q')$. 
If $\Delta_{K}(\zeta)\neq 0$ for any $p$-th root of unity $\zeta$ and 
 $\sigma(K) \not \equiv 0 \pmod 4$ then $p$ must be odd.
\end{corollary}

\begin{proof}
We prove the contraposition of the statement. 
Suppose that $p$ is even. 
Let us put $p=2m$ $(m \in \mathbb{Z})$ and let $\zeta = \exp(\frac{2\pi\sqrt{-1}}{p})$ be the primitive $p$-th root of unity.
Since $\sigma_{\omega}(K)=\sigma_{\overline{\omega}}(K)$, we have 
$\sigma_{\zeta^{i}}(K)=\sigma_{\zeta^{p-i}}(K)$. 
Therefore  the total $p$-signature is given by
\[\sigma(K,p)=\sum_{i=1}^{p-1} \sigma_{\zeta^{i}}(K) = 2\sum_{i=1}^{m-1} \sigma_{\zeta^{i}}(K) + \sigma_{\zeta^{m}}(K) = 2\sum_{i=1}^{m-1} \sigma_{\zeta^{i}}(K) + \sigma(K). \]
Let $S$ be a Seifert matrix for $K$. We note that for $\omega \in \mathbb{C}$ with $|\omega|=1$, 
\[(1-\omega)S + (1-\overline{\omega})S^{T}= (1-\overline{\omega})(S^{T}-\omega S)\] 
and hence the matrix $(1-\omega)S + (1-\overline{\omega})S^{T}$ is non-singular unless $\omega$ is a root 
of the Alexander polynomial $\Delta_{K}(t)=\det(S^{T}-tS)$. 
This shows that $\sigma_{\omega}(K) \equiv 0 \pmod 2$ unless $\Delta_{K}(\omega)=0$, and we get
\[ \sigma(K,p) \equiv \sigma(K) \pmod 4 .\]
On the other hand, it follows from Theorem~\ref{theorem:cassonobst} that 
$4a_{2}(K)(q+q')+8\lambda(\Sigma) = -\sigma(K,p)$.
%\sigma(K,p)$. 
Thus we conclude that 
\[ \sigma(K,p) \equiv \sigma(K) \equiv 0 \pmod 4.\]
\end{proof}

This corollary may suggest that if $\sigma(K) \neq 0$, then there exist no chirally cosmetic surgeries on $K$ along slopes with even numerators. 
In fact, (chirally) cosmetic surgeries with surgery slopes of even numerators seem difficult to find. See \cite{Ichihara} for related arguments on this problem.

\section{$SL(2,\mathbb{C})$ Casson invariant }\label{sec:SL(2,C)Casson}

In this section, we recall a surgery formula of the $SL(2,\mathbb{C})$ Casson invariant, denoted by $\lambda_{SL(2,\mathbb{C})}$, based on \cite{BodenCurtis}. 

The $SL(2,\mathbb{C})$ Casson invariant is defined in \cite{Curtis} as an invariant of closed 3-manifolds. 
Roughly speaking it counts the signed equivalence classes of representations of the fundamental group in $SL(2,\mathbb{C})$ which is analogous to the Casson invariant that counts the $SU(2)$ representations.
Unlike the Casson invariant,  the $SL(2,\mathbb{C})$ Casson invariant is independent from the orientation of the 3-manifold $M$. 
That is, $\lambda_{SL(2,\mathbb{C})}(M)=\lambda_{SL(2,\mathbb{C})}(-M)$ holds.
At first glance, this is a bit disappointing since we are interested in chirally cosmetic surgeries where the orientation plays a crucial role. 
Nevertheless, as we will see in Section \ref{section:application}, 
information of the $SL(2,\mathbb{C})$ Casson invariant will be quite useful to study chirally cosmetic surgeries, 
when we combine with constraints from other invariants sensitive to the orientation.

Let $K$ be a small knot in an integral homology 3-sphere $\Sigma$. 
In \cite{Curtis}, Curtis gave a surgery formula of the $SL(2,\mathbb{C})$ Casson invariant $\lambda_{SL(2,\mathbb{C})}$ as follows: 
There exist half-integers $E_0, E_1 \in \frac{1}{2} \mathbb{Z}_{\ge 0}$ 
depending only on $K$ such that, for every admissible slope $p/q$, we have
$$
\lambda_{SL(2,\mathbb{C})} ( \Sigma_{K} (p/q) ) 
= \frac{1}{2} \left\| p / q \right\| - E_{\sigma(p)}.
$$
Here $\left\| p / q \right\| $ is the total Culler-Shalen seminorm of the slope $p/q$ 
and $\sigma (p) =0$ if $p$ is even and $\sigma (p)=1$ if $p$ is odd.
This surgery formula immediately gives the following constraint on  purely and chirally cosmetic surgeries.

\begin{theorem}
\label{theorem:SL2Ccasson-1}
Let $K$ be a small knot in an integral homology 3-sphere $\Sigma$. 
If $\Sigma_K (p/q)\cong \pm \Sigma_K (p/q')$ and slopes $p/q$ and $p/q'$ are admissible, then $|| p/q||=||p/q'||$ holds.
\end{theorem}

In the following, we explain some of the terminologies appearing in the statement above. 

A knot $K$ in a closed 3-manifold $M$ is said to be \textit{small} if its exterior $E(K)$ does not contain essential 
(i.e., incompressible and not boundary-parallel) embedded closed surfaces. 

Let $E(K)$ be the exterior of a small knot $K$ in an integral homology 3-sphere $\Sigma$. 
Let $X(K)$ be the character variety of $\pi_1 (E(K))$, i.e., the set of characters of $SL (2, \mathbb{C})$ representations of $\pi_1 (E(K))$ which naturally has a structure of complex affine algebraic variety. 
Similarly, let $X(\partial E(K))$ be the character variety of the peripheral subgroup $\pi_{1}(\partial E(K))$.

For $\xi \in H_{1}(\partial E(K);\mathbb{Z})$, let $I_{\xi} : X(K) \to \mathbb{C}$ be a regular function defined by $I_{\xi}(\chi)=\chi(\xi)$ for $\chi \in X(K)$. Here we view $H_{1}(\partial E(K); \mathbb{Z})$ as a subgroup of $\pi_1(E(K))$ by the natural inclusion map $H_{1}(\partial E(K); \mathbb{Z})\cong \pi_{1}(\partial E(K)) \hookrightarrow \pi_1(E(K))$. Let $f_\xi\colon X(K) \to \mathbb{C}$ be the regular function defined by $f_\xi  = I_{\xi} -2$ for $\xi \in H_1(\partial E(K); \mathbb{Z})$. 

Let $r\colon X(K)\to X(\partial E(K))$ be the restriction map induced by $\pi_1(\partial E(K)) \to \pi_1(E(K))$. 
For a component $X_i$ of $X(K)$ with $\dim X_i =1$ and $\dim r(X_i) =1$, 
let $f_{i,\xi}\colon X_i \to \mathbb{C}$ be the regular function obtained by restricting $f_{\xi}$ to $X_i$. 

For the smooth, projective curve $\widetilde{X}_i$ birationally equivalent to $X_i$, each regular function on $X_i$ naturally %extends
determines a rational function on $\widetilde{X}_i$. 
We denote the natural extension of $f_{i,\xi}$ to $\widetilde{X}_i$ by $\tilde{f}_{i, \xi}\colon \widetilde{X}_i \to \mathbb{C} \mathbb{P}^1$. 
For such $X_i$, 
we define the seminorm  $\| \;\; \|_{i}$ on $H_1(\partial E(K); \mathbb{R})$ by naturally extending 
$ \| \xi \|_{i} =  \deg( \tilde{f}_{i,\xi}) $ for $\xi \in H_1(\partial M; \mathbb{Z})$.

Let $X^* (K)$ be the subspace of characters of irreducible representations and let $\{X_i\}$ be the collection of all one-dimensional components of $X(K)$ such that $\dim r(X_i) =1$ and $X_i \cap X^*(K) \ne \emptyset$. 

Suppose that $E(K)$ has a Heegaard splitting $W_1 \cup_F W_2$ 
with compression-bodies $W_1 , W_2$ and a Heegaard surface $F$, 
that is, $W_1 \cup W_2 = E(K)$ and $\partial W_1 = \partial W_2 = W_1 \cap W_2 = F$. 
Then we have $X (K) = X (W_1) \cap X (W_2) $ and $X^* (K) = X^* (W_1) \cap X^* (W_2) $. 
Let $m_i > 0$ be the intersection multiplicity of $X_i$ 
as a curve in the intersection $X^* (W_1) \cdot X^* (W_2)$ in $X(F)$.
Then we define the total Culler-Shalen seminorm of a slope $p/q$ as 
$$\| p/q \| =  \sum_i m_i \| p\mu + q \lambda \|_{i},$$ 
where $\mu$ denotes the meridian and $\lambda$ the preferred longitude of the knot $K$.

A slope $\gamma$ is said to be \textit{regular} 
if $\mathrm{ker} (\rho \circ i^*)$ is not the cyclic group generated by $\gamma \in \pi_1 (\partial E(K))$ 
for any irreducible representation $\rho : \pi_1 (M) \to SL (2, \mathbb{C})$ satisfying that 
\begin{enumerate}
\item
the character $\chi_\rho$ lies on a one-dimensional component $X_i$ of $X(K)$ such that $r(X_i)$ is one-dimensional, and 
\item
$\mathrm{tr}\, \rho(\alpha) = \pm 2$ for all $\alpha$ in the image of %$i^* : \pi_1 (\partial E(K)) \to  \pi_1 (E(K))$. 
 $\iota_* : \pi_1 (\partial E(K)) \to  \pi_1 (E(K))$ where $\iota:\partial E(K) \rightarrow E(K)$ is the inclusion map.
\end{enumerate}

A slope is called a \textit{boundary slope} if there exists an essential surface embedded in $E(K)$ with a nonempty boundary representing the slope. A boundary slope is said to be \textit{strict} if it is the boundary slope of an essential surface that is not the fiber of any fibration over the circle.

A slope $\gamma = p/q$ is said to be \textit{admissible} for a knot $K$ if 
\begin{enumerate}
\item 
$p/q$ is a regular slope which is not a strict boundary slope, and 
\item
no $p'$-th root of unity is a root of the Alexander polynomial of $K$, where $p' = p$ if $p$ is odd and $p' =p/2$ if $p$ is even.
\end{enumerate}

\section{Chirally cosmetic surgeries and boundary slopes}\label{sec:bdryslope}

In this section, we give a different but more informative formulation of a constraint on chirally cosmetic surgeries from the $SL(2,\mathbb{C})$ Casson invariant. 
Such an approach to use the $SL(2,\mathbb{C})$ Casson invariant was considered by Ichihara and Saito in \cite{IchiharaSaito}.

Let $K$ be a hyperbolic knot in an integral homology sphere $\Sigma$. 
Let $m$ be the number of boundary slopes for $K$, 
and we denote the set of boundary slopes for $K$ by $\mathcal{B}_K = \left\{ b_1 / c_1 ,\ldots , b_m/c_m \right\}$ 
with  $b_j / c_j \in \mathbb{Q}$, $b_j \ge 0$ for $1 \le j \le m$ and ${b_1}/{c_1} < \cdots < {b_m}/{c_m} $. 

It is known that a Culler-Shalen seminorm $||\;\; ||_{i}$ of a given slope $\gamma$ is written as the weighted sum 
of the distances between all the pairs of $\gamma$ and a boundary slope. 

\begin{proposition}\cite[Lemma 6.2]{BoyerZhang}, \cite[Lemma 2.2.3]{Mattman}
For a hyperbolic knot $K$, and a curve $X_i$ in the character variety $X (K)$,
there exist non-negative constants $a^i_j \ge 0$ depending only on $K$ such that 
$$ \left\|  p/q \right\|_i = 2 \sum_{j=1}^{m} a^i_j \Delta \left( p/q , b_j/c_j \right) $$
holds for each $i$. 
Here $\Delta (p/q,r/s):=|ps-rq|$ denotes the distance between slopes $p/q$ and $r/s$, that is, the minimal geometric intersection number of the representatives of $p/q$ and $r/s$. 
\end{proposition}

Thus the total Culler-Shalen norm is given by
\begin{align*}
\left\| p/q \right\|&=  \sum_{i} m_{i}\left\| p/q \right\|_{i} = 2 \sum_{i} m_{i} \left(\sum_{j=1}^{m} a^i_j \Delta \left(p/q , b_j / c_j \right)  \right)\\
&=\sum_{j=1}^{m}\left(2\sum_{i}a^{i}_{j}m_{i}\right)\Delta\left(p/q , b_j / c_j \right).
\end{align*}
By putting $w_{j}=2\sum_{i}a^{i}_{j}m_{i}$ we have the following useful formula for the total Culler-Shalen norm: 
\begin{equation}
\label{eqn:total-Culler-Shalen-norm}
\left\| p /q  \right\| = \sum_{j=1}^{m} w_{j}|pc_{j}-qb_{j}|. 
\end{equation}

Based on this formula, we get the following more informative version of Theorem \ref{theorem:SL2Ccasson-1}.

\begin{theorem}\label{theorem:boundaryslopeI}
Let $K$ be a hyperbolic small knot in an integral homology sphere $\Sigma$. 
Assume that $\Sigma_K (p/q) \cong \pm \Sigma_K (p/q')$, and the slopes $p/q$ and $p/q'$ are admissible with $p/q, p/q' \not\in [ b_1 / c_1 , b_m / c_m ]$. 
Then the following holds. 
\begin{enumerate}
\item[(i)] $qq'<0$, i.e., the signs of slopes are opposite.
\item[(ii)] There exists a constant $C$ depending only on $K$ such that $(q+q')/p =C$.
\item[(iii)] If all the boundary slopes for $K$ are non-negative as rational numbers, 
the constant $C$ in (ii) equals to $ \| \mu \| / 2 \| \lambda \|$, 
where $\mu$ denotes the meridional slope and  $\lambda$ the preferred longitudinal slope for $K$. 
\end{enumerate}
\end{theorem}

\begin{proof}
Suppose that $p/q, p/q' \not\in [ b_1 / c_1 , b_m / c_m ]$ are admissible and that  $\Sigma_K (p/q) \cong \pm \Sigma_K (p/q')$. 
For a pair of slopes $p/q$ and $p/q'$ with $p > 0$ and $q, q' \neq 0$, by Equation~\eqref{eqn:total-Culler-Shalen-norm}, the difference $\lambda_{SL(2,\mathbb{C})} (\Sigma_K(p/q) ) - \lambda_{SL(2,\mathbb{C})} (\Sigma_K(p/q') ) = \frac{1}{2}\left(||p/q||-||p/q'||\right)$ is calculated as follows: 

\begin{equation}\label{eq1}
\frac{1}{2} \left( \left\| p/q \right\| - \left\| p /q'  \right\| \right) 
=\frac{1}{2} \sum_{j=1}^{m} w_j \left( \left| p c_j - q b_j \right| - \left| p c_j -  q' b_j \right| \right) .
\end{equation}

It follows from Theorem~\ref{theorem:SL2Ccasson-1} that $||p/q||-|| p/q'||=0$. 

First we show (i). 
Consider the case that $p/q < b_1 / c_1$ and $p/q' < b_1 / c_1$, and show that it cannot happen. 
In this case, since $0$ is always a boundary slope, $b_1 / c_1\leq 0$ holds, and so $q,q'<0$ must hold. 
Then we have the following. 
\begin{align*}
0 &= \left\| p/q \right\| - \left\| p/q' \right\| =\sum_{j=1}^{m} w_j ( ( q b_j - p c_j ) - ( q' b_j - p c_j ) )=\sum_{j=1}^{m} w_j ( q - q' ) b_j \\
&=(q-q')  \sum_{j=1}^{m} w_j b_j =(q-q')  \sum_{j=1}^{m}w_j|1 \cdot b_j- 0 \cdot c_j|\\
& = (q-q') \left\| 0/1 \right\|. 
\end{align*}

Thus it implies that $||0/1||=0$. 
However, for a hyperbolic knot $K$, the total Culler-Shalen norm is actually a norm on $H_1 (\partial M ; \mathbb{R})$. 
In particular, $\| 0/1 \| \ne 0$. This is a contradiction. 

In the case of $b_m / c_m < p/q$ and $b_m /c_m < p/q'$, we will also get a contradiction by the same argument. 
Thus we may assume that (by changing the role of $q$ and $q'$ if necessary) $p/q < b_1 / c_1 < b_m / c_m < p/q'$. 
Since $b_1 / c_1 \leq 0 \leq  b_m / c_m$, this shows $q<0<q'$; hence the signs of $q$ and $q'$ must be opposite. 

Next we prove (ii). 
Since we may assume that $p/q < b_1 / c_1 < b_m / c_m < p/q'$, we get
\begin{align*}
0=  \left\| p/q \right\| - \left\| p/q' \right\| 
&=\sum_{j=1}^{m}w_j ( ( p c_j  - q b_j ) - ( q' b_j - p c_j ) )\\
&=\sum_{j=1}^{m} w_j (  2 p c_j - ( q + q' ) b_j ) \\
&= 2p\sum_{j=1}^{m}w_j c_j   -(q + q')\sum_{j=1}^{m} w_j b_j  \;.
\end{align*}

Consequently, we get the following equality whose right-hand side only depends on $K$: 
$$
\frac{ q + q' }{p} = \frac{ 2 \sum_{j=1}^{m} w_j c_j }{\sum_{j=1}^{m} w_j b_j  } .
$$

Finally we show (iii). If all the slopes are non-negative, then $c_{i}\geq 0$ for all $i$. 
Recall that $ \left\| \lambda \right\| = \left\|0/1 \right\| = 
\sum_{j=1}^{m} w_j b_j $. Since all the $c_j$'s are non-negative, 
$ \left\| \mu \right\| = \left\| 1/0 \right\| = \sum_{j=1}^{m} w_j | - c_j | = \sum_{j=1}^{m} w_j c_j $.  
We conclude that 
$$
\frac{ q + q' }{p} = \frac{ 2 \left\| \mu \right\|  }{ \left\| \lambda \right\|  } .
$$
\end{proof}

We also have the following, which gives an interesting relation between signs of boundary slopes and cosmetic surgeries.

\begin{theorem}\label{theorem:boundaryslopeII}
Let $K$ be a hyperbolic small knot in an integral homology sphere $\Sigma$. 
Assume that two slopes $p/q$ and $p/q'$ are admissible with $\Sigma_{K}(p/q)\cong \pm \Sigma_{K}(p/q') $.
If all the boundary slopes are non-negative (resp. non-positive), then $\frac{q+q'}{p}> 0$ (resp. $\frac{q+q'}{p} <0$).
\end{theorem}
\begin{proof}
We deal with the case all the boundary slopes are non-negative. 
The case that all the boundary slopes are non-positive can be dealt with in the same way. 

Since $\Sigma_{K}(p/q)\cong \pm \Sigma_{K}(p/q') $, we have $||p/q||=||p/q'||$.

If $p/q,p/q'>0$, then $\frac{q+q'}{p}>0$.
By Theorem \ref{theorem:boundaryslopeI} (iii), $\frac{p}{q}, \frac{p}{q'}<0$ cannot happen, 
because in such a case, we have a contradiction: $0>\frac{q+q'}{p} = \frac{2||\mu||}{|| \lambda||}>0$. 
Thus in the following, we may assume that $p/q'<0<p/q$, in particular, $q'<0<q$.

Let $N$ be the integer that satisfies $\frac{b_N}{c_{N}} < \frac{p}{q} < \frac{b_{N+1}}{c_{N+1}}$. 
When $\frac{b_{m}}{c_m} <\frac{p}{q}$, we define $N=m$.
Then by Equation~\eqref{eqn:total-Culler-Shalen-norm}, 
\[
\left\|p/q \right\| = \sum_{j=1}^{m} w_{j}\Delta \left(p/q, b_j/c_j\right)= \sum_{j=1}^{N} w_{j}(b_{j}q-pc_{j}) + \sum_{j=N+1}^{m}w_{j}(-b_{j}q+pc_{j}).
\]
On the other hand, we have 
\[ \left\|p/q' \right\| =\sum_{j=1}^{m}w_{j}(-b_{j}q'+pc_{j}). \]
Thus we obtain 
\begin{align*}
 &0=\sum_{j=1}^{N} w_{j}(b_{j}q-pc_{j}) + \sum_{j=N+1}^{m}w_{j}(-b_{j}q+pc_{j})- \sum_{j=1}^{m}w_{j}(-b_{j}q'+pc_{j}).
 \end{align*}
It follows that 
\begin{align*}
0= \sum_{j=1}^{N}w_{j}b_{j}(q+q') -2 \sum_{j=1}^{N} w_{j}c_{j}p + \sum_{j=N+1}^{m}w_{j}b_{j}(q'-q) 
\end{align*}
and so, 
\begin{align*}
\frac{q+q'}{p} \sum_{j=1}^{N}w_{j}b_{j}= 2 \sum_{j=1}^{N}w_j c_j +\frac{q-q'}{p}\sum_{j=N+1}^{m} w_{j} b_{j} \;.
\end{align*}
Since we are assuming all the boundary slopes are non-negative, i.e., %$c_{i}\geq 0$,
$c_i>0$ not all of $w_{i}$ are zero. Thus the right-hand side is always positive. This proves $\frac{q+q'}{p} > 0$.
\end{proof}

In general by calculating the total Culler-Shalen norm explicitly, we get more detailed constraint on $p,q,q'$ for the slopes of chirally cosmetic surgery $\Sigma^{3}_K(p/q) \cong -\Sigma^{3}_K(p/q')$. For later use, we illustrate one particular example.

\begin{example}[Constraint of chirally cosmetic surgery from $SL (2, \mathbb{C})$-Casson invariant for the $5_2$ knot]
\label{exam:5_2}
Let $K$ be the knot $5_2$ in $S^{3}$. 
Since no root of unity is a root of the Alexander polynomial $\Delta_K(t)=2t^{2}-3t+2$, all slopes except the boundary slopes $0,-4,-10$, are admissible. By \cite{BodenCurtis} the total Culler-Shalen seminorm $||p/q||$ of $K$ is given by the formula
\[ 
||p/q||=p+|p+4q|+|p+10q| = \begin{cases}
3p+14q & -\frac{p}{10}<q \\
p-6q    &  -\frac{p}{4} <q \leq -\frac{p}{10}\\
-p-14q & q\leq -\frac{p}{4}.
\end{cases} 
\]
Thus we have either
\begin{itemize}
\item $q>-\frac{p}{10}$, $-\frac{p}{4} <q' \leq -\frac{p}{10}$ and $p+7q=-3q'$, or,
\item $q>-\frac{p}{10}$, $q\leq -\frac{p}{4}$ and $7q+7q'=-2p$. 
\end{itemize}
\end{example}

\section{Other constraints for chirally cosmetic surgeries}
\label{section:others}
In this section, we review other known constraints for a knot to admit a chirally cosmetic surgeries.

\subsection{Degree two finite type invariant}

The Casson invariant, also known as the Casson-Walker invariant, is known to be regarded as the degree one part of the \textit{LMO invariant}. 
It is an invariant of rational homology 3-spheres which is universal among all the finite type invariants \cite{LeMurakamiOhtsuki}. 
In \cite{ItoTetsuya}, a constraint for purely and chirally cosmetic surgeries derived from higher degree parts of the LMO invariant was studied. 

Among them, the degree two part provides the following simple but useful obstruction for a knot to admit chirally cosmetic surgeries.
We remark that like the $SL(2,\mathbb{C})$ Casson invariant, the degree two part of the LMO invariant does not depend on orientations of 3-manifolds.
 
Let $v_{3}(K)$ be the primitive finite type invariant of degree three of a knot $K$ in $S^3$, normalized so that it takes the value %$-\frac{1}{4}$ 
$\frac{1}{4}$ on the right-handed trefoil. 
Using the derivatives of the Jones polynomial $V_{K}(t)$ of $K$, we see that $v_{3}(K)$ is written by
\[ %v_{3}(K)= -\dfrac{1}{36}V'''_{K}(1)-\frac{1}{12}V''_{K}(1), 
v_{3}(K)= -\dfrac{1}{144}V'''_{K}(1)-\frac{1}{48}V''_{K}(1), 
\]
and $v_{3}(K)$ satisfies the following skein relation \cite{IchiharaWu}: 
\[ v_{3}(K_{+})-v_{3}(K_{-})=-\frac{a_{2}(K')+a_{2}(K'')}{4}+\frac{a_{2}(K_{+})+a_{2}(K_{-})+(lk(K,K'))^{2}}{8}.\]
Here, $(K_{+},K_{-},K_{0})$ denotes the usual skein triple, where we view $K_{0}$ as a two-component link $K' \cup K''$.

Let $a_{2}(K)$ and $a_{4}(K)$ be the coefficients of $z^{2}$ and $z^{4}$ in the Conway polynomial of $K$ respectively. 
From the degree two part of the LMO invariant, we have the following constraint on  a knot to admit chirally cosmetic surgeries.

\begin{theorem}\cite[Corollary 1.3]{ItoTetsuya}
\label{theorem:v3}
Assume that $p/q$- and $p/q'$-surgeries on a knot $K$ in $S^3$ give orientation-reversingly homeomorphic 3-manifolds.
\begin{enumerate}
\item[(i)] If $q=-q'$, then $v_{3}(K)=0$.
\item[(ii)] If $q \neq -q'$, then $v_{3}(K)\neq 0$ and 
\[ \frac{p}{q+q'}= \frac{7a_{2}(K)^{2}-a_{2}(K)-10a_{4}(K)}{8v_3(K)}.\]
\end{enumerate}
\end{theorem}

It is interesting to compare Theorem \ref{theorem:v3} (ii) with Theorem \ref{theorem:boundaryslopeI} (ii), (iii) and Theorem \ref{theorem:boundaryslopeII}: 
they provide unexpected relations among $v_{3}(K)$, the Culler-Shalen norm and boundary slopes for chirally cosmetic surgeries.

\subsection{Heegaard Floer homology}

Recent progress on Heegaard Floer homology theory provides the following strong constraint on  a knot to have a purely or chirally cosmetic surgeries. 
%Here we just include the following result for the readers' convenience. 

\begin{theorem}
\label{theorem:HFK}
{$ $}
\begin{enumerate}
\item[(i)] If $S^{3}_{K}(r) \cong S^{3}_{K}(r')$, then $r'=\pm r$ \cite{NiWu} .
\item[(ii)] If $S^{3}_{K}(r) \cong \pm S^{3}_{K}(r')$ for $r,r'$ with $ r r'>0$ (i.e. if $K$ admits chirally cosmetic surgeries along the slopes with the same sign), then $S^{3}_{K}(r)$ is an L-space  \cite[Theorem 1.6]{OzsvathSzabo11}, in particular, by \cite{Ni} $K$ is fibered. 
Moreover, if the genus of $K$ is one, then $K$ is the trefoil \cite{WangJiajun}. 
\end{enumerate}
\end{theorem}

Theorems \ref{theorem:HFK} (ii) and \ref{theorem:boundaryslopeI} (i) seem to suggest that chirally cosmetic surgeries with the same sign are quite limited.

\section{Chirally cosmetic surgeries on positive knots or alternating knots}
\label{section:application}

First of all, we combine Casson invariant constraint and degree two finite type invariant constraint to obtain following criterion for non-existence of chirally cosmetic surgery.

\begin{theorem}
\label{theorem:criterion}
Let $K$ be a knot and let $d(K)$ be the degree of the Alexander polynomial of $K$. If $S^{3}_{K}(p\slash q) \not\cong -S^{3}_{K}(p\slash q')$ for some $p/q,p/q'$ with $q+q'\neq 0$, then 
\[ 4|a_2(K)| \leq d(K) \left| \frac{7a_2^{2}(K)-a_2(K)-10a_4(K)}{8v_3(K)} \right|\]
\end{theorem}
\begin{proof}
By Theorem \ref{theorem:cassonobst} we have $12(q+q')a_2(K)=-3\sigma(K,p)$.
Since $|\sigma_{\omega}(K)| \leq d(K)$ for all $\omega \in \{ z \in \mathbb{C}\: | \: |z|=1\}$,
 $|12(q+q')a_2(K)| \leq 3p d(K)$. Therefore by Theorem \ref{theorem:v3}
\[ 4|a_2(K)| \leq d(K) \left| \frac{p}{q+q'}\right| = d(K) \left| \frac{7a_2^{2}(K)-a_2(K)-10a_4(K)}{8v_3(K)} \right| \]
\end{proof}

For positive knots we may make a criterion in a simpler (but a weaker) form.
\begin{corollary}
\label{cor:positiveknot}
Let $K$ be a non-trivial positive knot. If $\frac{16}{7} \geq g(K)\frac{a_2(K)}{v_3(K)}$  Then $K$ does not admit chirally cosmetic surgery where $g(K)$ denotes the genus of $K$.
\end{corollary}
\begin{proof}
We use the following properties of a positive knot $K$.
\begin{enumerate}
\item[(a)] The total $p$-signature of $K$ is always negative: $\sigma(K,p)<0$. 
\item[(b)] All the coefficients of the Conway polynomial of $K$ are non-negative. In particular, $a_{2}(K),a_4(K) \geq 0$ \cite[Corollaries 2.1 and 2.2]{Cromwell}.
\item[(c)] The degree of the Alexander polynomial is twice of the genus of $K$; $d(K)=2g(K)$.
\item[(d)] $v_3(K)>0$ \cite{Stoimenow-positive}.
\end{enumerate}
Assume $S^{3}_{K}(p\slash q) \cong -S^{3}_{K}(p\slash q')$.
By (a),(b) and Theorem \ref{theorem:cassonobst} we have $12(q+q')a_2(K) =-3 \sigma(K,p)>0$ so $q+q'>0$.
Then by Theorem \ref{theorem:v3} and (c),(d) we get 
\[ 2a_2(K) \leq g(K) \frac{7a_2^{2}(K)-a_2(K)-10a_4(K)}{8v_3(K)} < \frac{7a_2^{2}(K)}{8v_3(K)}g(K).\]
\end{proof}

As a first application of our criterion, we show that under a mild assumption, knots obtained by adding sufficiently many twists never admit a chirally cosmetic surgery. 

\begin{theorem}
\label{theorem:larget_N}
Let $K$ be a knot represented by a diagram $D$.
At a positive crossing $c$ of $D$ let $L=K' \cup K''$ be the 2-component link obtained by resolving the crossing $c$, and for $N>0$, let $K_{N}$ be the knot obtained from $K$ by applying $\overline{t}_{N+1}$-move at $c$, which replaces the positive crossing $c$ with consecutive $(2N+1)$ positive crossings (see Figure \ref{fig2}). 
If $lk(K',K'')\neq 0$, then for a sufficiently large $N$, $K_N$ does not admit chirally cosmetic surgery.
\end{theorem}

\begin{figure}[htb]
 \begin{center}
  {\unitlength=1cm
  \begin{picture}(8.89,1.66)
   \put(0,0){\includegraphics{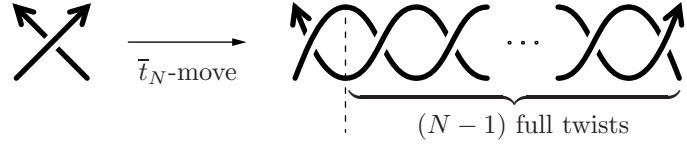}}
  \put(4.5,.6){$\underbrace{\hspace{4.4cm}}$}
  \put(5.4,0){$(N-1)$ full twists}
  \put(1.7,0.7){$\overline{t}_N$-move}
  \end{picture}}
  \caption{The $\overline{t}_N$-move: It replaces a positive crossing $c$ with consecutive $2N-1$ positive crossings.} 
  \label{fig2}
 \end{center}
\end{figure}

\begin{proof}
Let $\nabla_{L}(z)=a_1(L)z +a_3(L)z^{3}+\cdots = lk(K',K'')z +a_3(L)z^{3}+\cdots$.
By the skein formula we have  
\begin{align*}
a_2(K_N) & =a_2(K)+ lk(K',K'')N\\
v_3(K_N) &= v_3(K) - \frac{1}{4}(a_2(K')+a_2(K''))N + \frac{1}{8}\left(2 a_2(K) +lk(K',K'')^{2} \right)N \\
& \quad+ \frac{1}{8}lk(K',K'')^{2}N^{2}\\
a_4(K_N) &=a_4(K) + a_3(L) N 
\end{align*}
Since we have assumed $lk(K',K'')\neq 0$, $a_2(K_N) \neq 0$ and $v_3(K_N)\neq 0$ for sufficiently large $N$. In particular, by Theorem \ref{theorem:v3} $K_N$ does not admit a chirally cosmetic surgery $S^{3}_{K_N}(p/q)\cong -S^{3}_{K_N}(p/q')$ with $q+q'=0$.

Moreover, by these formula we have 
\[ \left| \frac{7a_2^{2}(K_N)-a_2(K_N)-10a_4(K_N)}{32 a_{2}(K_N)v_3(K_N)} \right| = 
 \left| \frac{c_0+c_1 N+ c_2N^{2}}{d_0+d_1N + d_2N + d_3N^{3}} \right|\]
where $c_0,c_1,c_2,d_0,d_1,d_2,d_3$ are constant that do not depend on $N$ and $d_3=4 lk(K',K'')^{3} \neq 0$. 
Let $g(D)$ be the genus of the diagram, the genus of the Seifert surface of $K$ obtained by Seifert's algorithm. Since $\overline{t}_{N}$-move preserves the diagram genus $d(K_N) \leq 2g(K_N) \leq 2g(D_N) = 2g(D)$, where $g(D_N)$ denotes the diagram genus of the diagram of $K_N$. Hence
\begin{align*}
\lim_{N\to \infty} d(K_N) \left| \frac{7a_2^{2}(K_N)-a_2(K_N)-10a_4(K_N)}{32 a_{2}(K_N)v_3(K_N)} \right| = 0
\end{align*}
Therefore by Theorem \ref{theorem:criterion} if $N$ is sufficiently large, it does not admit chirally cosmetic surgery.
\end{proof}

Finally we give a complete classification of chirally cosmetic surgeries of alternating knot of genus one.

\begin{theorem}
\label{theorem:main-genus-one-alternating}
Let $K$ be an alternating knot of genus one. 
For distinct slopes $r$ and $r'$, if the $r$- and $r'$-surgeries on $K$ are chirally cosmetic, 
then either
\begin{enumerate}
\item[(i)] $K$ is amphicheiral and $r=-r'$, or,
\item[(ii)] $K$ is the positive or the negative trefoil, and 
\[ \{r,r'\} = 
\left\{\frac{18k+9}{3k+1}, \frac{18k+9}{3k+2} \right\},\ \mathit{or}\ 
\left\{- \frac{18k+9}{3k+1}, - \frac{18k+9}{3k+2} \right\} \quad (k \in \mathbb{Z}).\]
\end{enumerate}
\end{theorem}

The classification of cosmetic surgeries on the trefoil knot seems to be known before; See Corollary \ref{corollary:torus-knot} in Appendix, where we give a complete list of cosmetic surgeries on torus knots, based on the result in \cite{Rong}. 
Also, the classification of chirally cosmetic surgeries on amphicheiral knots follows from Theorem \ref{theorem:HFK} (i); 
If $K$ is amphicheiral, $S^{3}_K(r)\cong -S^{3}_K(-r)$ for all slope $r$. Therefore, when $K$ is amphicheiral and $r,r'$ are distinct slopes such that $S^{3}_{K}(r)\cong -S^{3}_{K}(r')$, then $S^{3}_{K}(r)\cong -S^{3}_{K}(r')\cong S^3_K(-r')$. By Theorem \ref{theorem:HFK} (i), this implies $r= \pm (-r')$ so $r=-r'$. 

The main content of Theorem \ref{theorem:main-genus-one-alternating} is to show the non-existence of chirally cosmetic surgeries on other alternating knots of genus one, which will be achieved in the sequel. 

It is known \cite{Stoimenow} that alternating knot of genus one is either
\begin{itemize}
\item The double twist knot $J(\ell,m)$ with even $\ell, m$ (see Figure \ref{fig1}), or,
\item The 3-pretzel knot $P(p,q,r)$ with odd $p,q,r$ having the same sign. 
\end{itemize}

We remark that the set of alternating knots of genus one includes the set of positive knots of genus one, which are the 3-pretzel knot $P(p,q,r)$ with odd $p,q,r>0$ \cite{Stoimenow}.

\begin{figure}[htb]
 \begin{center}
  {\unitlength=1cm
  \begin{picture}(12,7.6)
   \put(0,0){\includegraphics{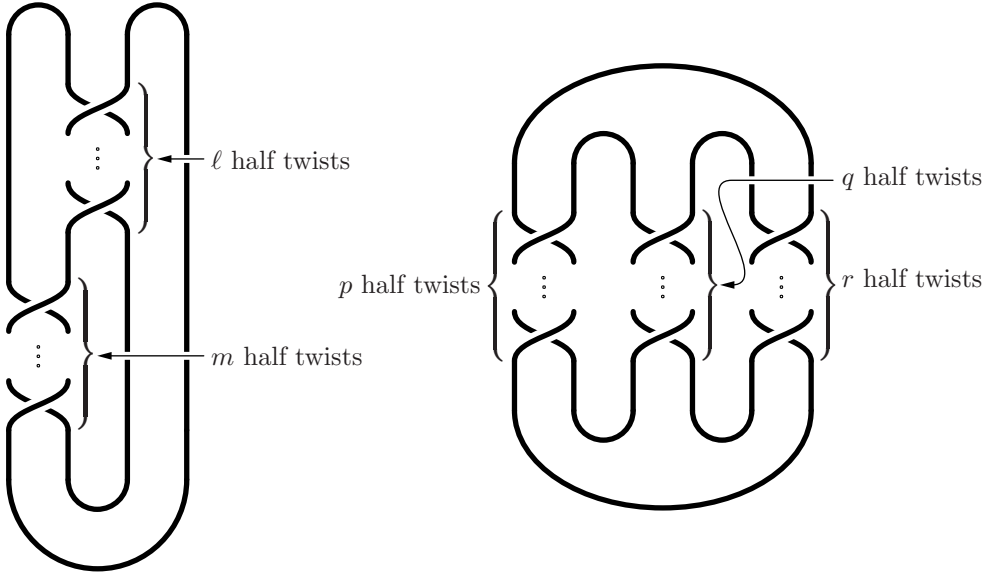}}
  \put(1.8,4.5){\rotatebox{90}{$\underbrace{\hspace{2cm}}$}}
  \put(2.8,5.35){$\ell$ half twists}
  \put(1,1.9){\rotatebox{90}{$\underbrace{\hspace{2cm}}$}}
  \put(2.8,2.75){$m$ half twists}
  \put(6.45,4.8){\rotatebox{-90}{$\underbrace{\hspace{2cm}}$}}
  \put(4.5,3.7){$p$ half twists}
  \put(9.3,2.81){\rotatebox{90}{$\underbrace{\hspace{2cm}}$}}
  \put(11.18,5.15){$q$ half twists}
  \put(10.87,2.81){\rotatebox{90}{$\underbrace{\hspace{2cm}}$}}
  \put(11.18,3.78){$r$ half twists}
  \end{picture}}
  \caption{The left figure shows the double twist knot $J(\ell,m)$. The knot $J(\ell,-m)$ is obtained by 
twisting the left-sided two strands ``negatively''. The right figure shows the 3-pretzel knot $P(p,q,r)$.} 
  \label{fig1}
 \end{center}
\end{figure}

First of all, we compute the Conway polynomial and the invariant $v_{3}$.
Although these formulae seem to be appeared in several places (in \cite{IchiharaWu}, for example), here we give a direct computation for readers' convenience.

\begin{proposition}
\label{proposition:computation}
Let $K$ be a genus one alternating knot.
\begin{enumerate}
\item If $K$ is the double twist know $J(\ell,m)$ with even $\ell,m$,
\[ \nabla_{K}(z)= 1+ \frac{\ell m}{64}z^{2}, \quad v_{3}(K)= -\frac{m \ell}{64}(\ell+m). \]
\item If $K$ is the 3-pretzel knot $P(p,q,r)$ with positive odd $p,q,r,$ then 
\begin{align*} \nabla_{P(p,q,r)}(z)& = 1+ \frac{1}{4}\left( pq+qr+rp+1 \right)z^{2},\\ v_{3}(P(p,q,r))&= \frac{1}{64}\left( (p+q+r+1)(pq+qr+rp+1)+(p-1)(q-1)(r-1) \right)
\end{align*}
\end{enumerate}
\end{proposition}

\begin{proof}
By the skein relation of the Conway polynomial, we have
\[
 \nabla_{J(\ell-2, m)}(z) -\nabla_{J(\ell,m)}(z) = z \nabla_{T(2, - m)}(z), 
\]
where $T(2,m)$ denotes the $(2,m)$-torus link. 
Here we apply the skein relation at a crossing in the vertical $\ell$ half twists in Figure \ref{fig1}.
Since $\nabla_{T_{2,m}}(z)=\frac{m}{2}z$, we have  
\[
\nabla_{J(\ell, m)}(z) = \nabla_{J(0, m)} (z) - \frac{\ell}{2} z \nabla_{(T(2,-m))}(z) = 1+\frac{\ell m}{4}z^{2}.
\]
Similarly, by the skein relation of $v_{3}(K)$, we have 
\begin{align*}
v_{3}(J(\ell-2,m) - v_{3}(J(\ell,m)) &= \frac{1}{8} \left(a_{2}(J(\ell-2,m)) +a_{2}(J(\ell,m))+\frac{m^{2}}{4} \right) \\
&= \frac{1}{8}\left( \frac{(\ell-2)m}{4} + \frac{\ell m}{4}+\frac{m^{2}}{4} \right)= \frac{m}{32}(2\ell-2+m) . 
\end{align*}
Hence we obtain 
\[  v_{3}(J(\ell,m))=-\sum_{i=1}^{\frac{\ell}{2}} \frac{m}{32}(4i+m-1) =-\frac{\ell m}{64}(\ell + m). 
\] 

The computation for $P(p,q,r)$ is similar; for $p \geq 3$, by the skein relation applied at a crossing in the vertical $p$ half twists we have $\nabla_{P(p,q,r)}(z) - \nabla_{P(p-2,q,r)}(z)= z\nabla_{T(2,q+r)}(z)$
so
\[ \nabla_{P(p,q,r)}(z) = \nabla_{P(1,q,r)}(z)+\frac{(p-1)(q+r)}{4}z\]
Noting $K(p,q,r)\cong K(q,r,p) \cong K(r,p,q)$ and $K(1,1,1)$ is right-handed trefoil, we inductively compute $\nabla_{P(p,q,r)}(z)$ as
\begin{align*}
\nabla_{P(p,1,1)}(z) &= 1+z + \frac{p-1}{2}z = \frac{(p+1)}{2}z,\\
\nabla_{P(p,1,r)}(z) &= \nabla_{P(1,1,r)}(z) + \frac{(p-1)(r+1)}{4}z = 1+ \frac{(p+1)(r+1)}{4}z, 
\end{align*}
and
\begin{align*}
\nabla_{P(p,q,r)}(z) &=\nabla_{P(1,q,r)}(z)+\frac{(p-1)(q+r)}{4}z \\
&= 1+ \frac{(q+1)(r+1)}{4}z + \frac{(p-1)(q+r)}{4}z =1+\frac{pq+qr+rp+1}{4}z.
\end{align*}

As for the $v_3(P(p,q,r))$, by the skein formula we have
\[ v_{3}(P(p,q,r)) = v_{3}(P(1,q,r)) +\frac{1}{8}\sum_{i=1}^{\frac{p-1}{2}}\left( a_{2}(P(2i+1,q,r))+a_2(P(2i-1,q,r))+\frac{1}{4}(q+r)^{2}\right) \]
hence by a similar argument we have 
\begin{align*}
v_{3}(P(p,1,1)) &= \frac{1}{4}+ \frac{1}{8}\sum_{i=1}^{\frac{p-1}{2}}\left(\frac{4i+4}{4}+\frac{4i}{4}+\frac{(1+1)^{2}}{4}\right) = \frac{1}{32}(p+1)(p+3),\\
v_{3}(P(p,1,r)) & = \frac{1}{32}(r+1)(r+3)+ \frac{1}{8}\sum_{i=1}^{\frac{p-1}{2}}\left( \frac{(2i+2)(r+1)}{4}+\frac{(2i)(r+1)}{4} +\frac{(r+1)^{2}}{4}\right)\\
&= \frac{1}{64}(p+1)(r+1)(p+r+2),\\
v_{3}(P(p,q,r)) & = \frac{1}{64}(q+1)(r+1)(q+r+2)\\
& \quad + \frac{1}{8}\sum_{i=1}^{\frac{p-1}{2}}\left( \frac{(2i+1)(q+r)+qr+1}{4}+\frac{(2i-1)(q+r)+qr+1}{4} +\frac{(q+r)^{2}}{4}\right)\\
&= \frac{1}{64}\left( (p+q+r+1)(pq+qr+rp+1)+(p-1)(q-1)(r-1) \right).\\
\end{align*}
\end{proof}

We remark that for a double twist knot $J(\ell,m)$, the finite type invariant $v_{3}$ completely detects the amphicheiral property.

\begin{corollary}\label{cor:J(l,m)amphicheiral}
The knot $J(\ell,m)$ is amphicheiral if and only if $m=-\ell$.
\end{corollary}
\begin{proof}
It can be directly seen from that $J(\ell,m)$ is amphicheiral if $m=-\ell$. 
Conversely, for an amphicheiral knot $K$, $v_{3}(K)=0$ holds, since $v_{3}(K)=-v_{3}(K!)$ for the mirror image $K!$ of $K$. 
Thus by Proposition \ref{proposition:computation}, we obtain $m=-\ell$.

\end{proof}

\begin{proof}[{Proof of Theorem~\ref{theorem:main-genus-one-alternating}}]
We show that a genus one alternating knot $K$, unless it is the trefoil or an amphicheiral knot, does not admit chirally cosmetic surgeries. \\

\underline{\textbf{Case 1:} $K=J(\ell,m), \ \ell \geq m>0$ (even)}\\

In this case $J(\ell,m)$ is a negative knot whose mirror image $J(-\ell,-m)$ is a positive knot. 
By Proposition \ref{proposition:computation} $\frac{a_{2}(J(-\ell,-m))}{v_3(J(-\ell,-m))}= \frac{16}{\ell+m}$ hence by Corollary \ref{cor:positiveknot} unless $(\ell,m)=(4,2)$ or $(\ell,m)=(2,2)$, $J(\ell,m)$ does not admit chirally cosmetic surgeries.

 The case $(\ell,m)=(2,2)$ is the trefoil so we show that $J(4,2)$ does not admit chirally cosmetic surgeries. To treat this case we use the $SL(2;\mathbb{C})$-Casson invariant and Heegaard Floer homology argument.

Assume, to the contrary that for $K=J(4,2)$ we have $S^{3}_K(p/q)\cong -S^{3}_K(p/q')$ for some $p/q\neq p/q'$. By Theorem \ref{theorem:v3} and Proposition \ref{proposition:computation}, we have $\frac{p}{q+q'}=-\frac{13}{3}$.

On the other hand, $J(2,4)$ is the $5_2$ knot and by a computation in Example \ref{exam:5_2}, we know that $p+7q=-3q'$ or $5q+5q'=-p$. If $p+7q=-3q'$, then we combining the equality $\frac{p}{q+q'}=-\frac{13}{3}$ we get $q'=2q$ so $q'$ and $q$ must have the same sign. By Theorem 5.2 (ii), this shows that $K$ is fibered, which is a contradiction. If $5q+5q'=-p$, then $\frac{p}{q+q'}=-5$ which contradicts $\frac{p}{q+q'}=- \frac{13}{3}$.\\

\underline{\textbf{Case 2:} $K=J(\ell,m), \ \ell <0 < m$ (even)}\\

We may actually assume $m \neq -\ell$, otherwise $K$ is amphicheiral by Corollary~\ref{cor:J(l,m)amphicheiral}. 
We see from Proposition~\ref{proposition:computation} that $a_{2}(K) = \frac{\ell m }{4} <0$. 
Thus, $\Delta_{K}(t)$ has no roots on the unit circle $ \{z \in \mathbb{C}\: | \: |z|=1\}$. This implies that $\sigma_{\omega}(K)=0$ for all $\omega \in \{z \in \mathbb{C}\: | \: |z|=1\} $. In particular, $\sigma(K,p)=0$ for all $p$. 

By Proposition~\ref{proposition:computation}, we have $a_2(K) \ne 0$. 
Then, by Theorem~\ref{theorem:cassonobst}, we obtain that $\frac{q+q'}{p}=0$. 
However, by Proposition \ref{proposition:computation}, $v_{3}(K)=v_{3}(J(\ell,m)) = -\frac{\ell m}{64}(\ell +m)\neq 0$. This contradicts Theorem~\ref{theorem:v3} (ii). \\

\underline{\textbf{Case 3:} $K=P(p,q,r), \ 0<p,q,r$ (odd)}\\

In this case $K$ is a positive knot and by Proposition~\ref{proposition:computation} $\frac{a_2(P(p,q,r))}{v_3(P(p,q,r))}<\frac{16}{p+q+r+1}$. By Corollary \ref{cor:positiveknot} unless $p+q+r < 7$, $P(p,q,r)$ does not admit chirally cosmetic surgeries. If $p+q+r=3$ it is the trefoil.
When $p+q+r=5$, $P(p,q,r)$ is equal to the mirror image of $J(2,4)$ which we have treated in Case 1.
\end{proof}

Our proof of Theorem~\ref{theorem:main-genus-one-alternating} demonstrates a general strategy of establishing the non-existence of chirally cosmetic surgeries for a family of knots defined by diagrams parametrized by the number of twisting, like Figure \ref{fig2}. 

In a light of Theorem \ref{theorem:larget_N}, by computing $a_2$, $a_4$ and $v_3$ we may exclude the possibility of chirally cosmetic surgery for many cases. Then we may exclude the remaining cases by using $SL(2;\mathbb{C})$-Casson invariant and Heegaard Floer homology argument.

\appendix

\section{Cosmetic surgeries on Torus knots}\label{App}

Here, based on \cite{Rong}, we give a complete classification of chirally cosmetic surgeries on torus knots in $S^3$. 
For the notation used here, see \cite{Rong} in detail. 

Let $X$ be a compact orientable Seifert fibered space with an incompressible torus boundary and orientable base orbifold. 
Let $R$ be  a cross-section of an $S^1$-bundle over a surface obtained from $X$ by removing open tubular neighborhoods of the singular fibers of the Seifert fibration. 
Set $c=R\cap \partial X$ and let $h$ be a regular fiber on $\partial X$. 
Once $c$ is fixed, a slope $\gamma$ on $\partial X$ is described by a rational number $\beta/\alpha$ 
such that $[\gamma]=\alpha[c] - \beta[h] \in H_1(\partial X, \mathbb{Z})$. 
Remark that, when $X$ is a knot complement in $S^3$, this rational number is different from that obtained by using the standard meridian-longitude system. 

\begin{theorem}{{\cite[Theorem 1]{Rong}}}\label{Rong}
Let $X$ be as above and $X(\gamma_i)$ ($i=1,2$) $3$-manifolds obtained by Dehn filling of $X$ along slopes $\gamma_1$ and $\gamma_2$ respectively. 
Suppose that $X(\gamma_1)\cong \pm X(\gamma_2)$ and there is no homeomorphism $f:X\to X$ sending $\gamma_1$ to $\gamma_2$. 
Then 
\begin{enumerate}
\item 
$X(\gamma_1)\cong -X(\gamma_2)$ and $X(\gamma_1)\not\cong +X(\gamma_2)$, and 
\item 
under some choice of the section $c$ on $\partial X$, the Seifert invariant of $X$ can be 
written as 
$$
\left\{ \epsilon,g; \frac{1}{2}, \ldots, \frac{1}{2},r_2,-r_2,\ldots,r_k,-r_k,r_1 \right\}, 
\ \mathit{where}\ r_i\not\equiv 0,\frac{1}{2}\ (\bmod\ 1). 
$$
Moreover, $\gamma_1$ and $\gamma_2$ are determined by rational numbers $-r_1+m$ and $-r_1-n-m$ with respect to the Seifert fibration, respectively, where $n$ is the number of singular fibers of type $\frac{1}{2}$ in the above and $m\ne -\frac{n}{2}$ is an integer. 
\end{enumerate}
Conversely, if $\gamma_1$ and $\gamma_2$ are as in $(2)$ above, then $X(\gamma_1)\cong -X(\gamma_2)$, $X(\gamma_1)\not\cong X(\gamma_2)$ and there is no homeomorphism $f:X\to X$ sending $\gamma_1$ to $\gamma_2$. 
\end{theorem}

\begin{corollary}\label{corollary:torus-knot}
Let $T_{r,s}$ be the $(r,s)$-torus knot. 
If $T_{r,s}(p/q)\cong \pm T_{r,s}(p/q')$, then $T_{r,s}(p/q)\cong -T_{r,s}(p/q')$, $s=2$ and 
$$
p/q=\frac{2r^2(2m+1)}{r(2m+1)+1}, \hspace{.5cm} 
p/q'=\frac{2r^2(2m+1)}{r(2m+1)-1} 
$$
for a positive integer $m$. 
Conversely, if $p/q$ and $p/q'$ are such rational numbers, then $T_{r,2}(p/q)\cong -T_{r,2}(p/q')$ for any odd integer $r\ge 3$. 
\end{corollary}

\begin{proof}
Let $E(K)$ be the exterior of the $(r,s)$-torus knot $K=T_{r,s}$.
With respect to the standard meridian-longitude system, the slope of the regular fiber $h$ of the Seifert fibration of $E(K)$ is described as $rs \in \mathbb{Q}$. 
According to the sign convention in \cite{Rong}, we have $[h]=-(rs)[\mu] - [\lambda]$ in $H_{1}(\partial E(K); \mathbb{Z})$. 
We take a cross-section $R$ so that $c$ represents the slope $rs-1$ with respect to the meridian-longitude system. 
Precisely, we take $R$ so that $[c]=(rs-1)[\mu] +[\lambda]$ in $H_{1}(\partial E(K); \mathbb{Z})$ for the meridional slope $\mu$ and the preferred longitudinal slope $\lambda$. 
Then the Seifert invariant of $E(K)$ is obtained as 
$\{ o_1,0; \frac{t}{r},\frac{u}{s}\}$, where $t,u\in \mathbb{Z}$ are integers taken so that $ru-ts=1$. 

Hence it follows from Theorem \ref{Rong} that if $T_{r,s}(p/q)\cong T_{r,s}(p/q')$, then $s=2, u=1$. 
Consequently 
its Seifert invariant is $\{ o_1,0; \frac{r-1}{2r},\frac{1}{2}\}$, and the surgery slope is described as $\gamma_1=-\frac{r-1}{2r}+m$ and $\gamma_2=-\frac{r-1}{2r}-1-m$ with $m \in \mathbb{Z}$ with respect to the Seifert fibration. 

Let $\gamma_1 = p/q$ and $\gamma_2=p/q'$ with respect to the standard meridian-longitude system. 
Then we have
\begin{align*}
\begin{cases}
2r[c]-(1-r+2rm)[h]=p[\mu]+q[\lambda]\\
2r[c]-(1-3r-2rm)[h]=p[\mu]+q'[\lambda] , 
\end{cases} 
\end{align*}
and so, 
\begin{align*}
\begin{cases}
(2r(2r-1)+(1-r+2rm)2r)[\lambda]+ (2r+(1-r+2rm))[\mu] =p[\lambda]+q[\mu]\\
(2r(2r-1)+(1-3r-2rm)2r)[\lambda] +(2r+(1-3r-2rm))[\mu]=p[\lambda]+q'[\mu]. 
\end{cases} 
\end{align*}
Therefore we conclude
\[ p/q=\frac{2r^2(2m+1)}{r(2m+1)+1}, \ 
p/q'=\frac{2r^2(2m+1)}{r(2m+1)-1}.  \]
as desired.
\end{proof}

\section*{Acknowledgements}
Ichihara is partially supported by JSPS KAKENHI Grant Number 26400100. 
Ito is partially supported by JSPS KAKENHI Grant Numbers 15K17540, 16H02145. 
Saito is partially supported by JSPS KAKENHI Grant Number 15K04869.

\end{document}